\documentclass[12pt]{amsart}
\usepackage{main}

\def\bar{\overline}
\def\hat{\widehat}
\def\tilde{\widetilde}
\def\D{\mathcal D}
\def\inter{\text{int}}

\def\det{\text{det}}

\def\E{\mathcal E}
\def\D{\mathcal D}

\def\expec{\mathbb{E}}

\def\t{\intercal}

\def\I{\beta}
\def\pgb{{\bf q}} 
\def\pt{{\bf p}} 

\title[Inverse problem with a white noise source]{Inverse problem for the wave equation with a white noise source}
\author{Tapio Helin}
\author{Matti Lassas}
\address{Department of Mathematics and Statistics, 
P.O. Box 68 ,
FI-00014 University of Helsinki.
Email: tapio.helin$@$helsinki.fi,
matti.lassas$@$helsinki.fi.}
\author{Lauri Oksanen}
\address{University of Washington,
Department of Mathematics,
Box 354350,
Seattle, WA 98195-4350.
Email: lauri.oksanen$@$math.washington.edu.}

\date{\today}

\begin{document}

\begin{abstract}
We consider a smooth Riemannian metric tensor $g$ on $\R^n$
and study the stochastic wave equation for the Laplace-Beltrami operator
$\p_t^2 u - \Delta_g u = F$.
Here, $F=F(t,x,\omega)$ is a random source that has white noise distribution supported on the boundary of some smooth compact domain $M \subset \R^n$. 
We study the following formally posed inverse problem with only one measurement.
Suppose that $g$ is known only outside of a compact subset of $M^{int}$
and that a solution $u(t,x,\omega_0)$ is produced by a single realization of the source $F(t,x,\omega_0)$. We ask
what information regarding $g$ can be recovered by measuring $u(t,x,\omega_0)$ on $\R_+ \times \p M$? 
We prove that such measurement together with the realization of the source determine the scattering relation of the Riemannian manifold $(M, g)$ with probability one. That is, for all geodesics passing through $M$, the travel times together with the entering and exit points and directions are determined. In particular, if $(M,g)$ is a simple Riemannian manifold and $g$ is conformally Euclidian in $M$, the measurement determines the metric $g$ in $M$.
\end{abstract}

\maketitle

 
\section{Introduction}

We consider the wave equation
\begin{align}
\label{eq:problem}
&\p_t^2 u(t,x) - \Delta_g u(t,x) = F(t,x) \quad \text{on $(0, \infty) \times \R^n$},
\\
&u|_{t = 0} = \p_t u|_{t=0} = 0,\nonumber
\end{align}
where $n\geq 2$ and $\Delta_g$ is the Laplace--Beltrami operator
corresponding to a smooth time-independent Riemannian metric $g(x) = [g_{jk}]^n_{j,k=1}$, that is,
\begin{equation*}
	\Delta_g u = \sum_{j,k=1}^n |g|^{-1/2} \frac{\p}{\p x^j} \left( |g|^{1/2} g^{jk} \frac{\p}{\p x^k} u\right),
\end{equation*}
where $|g| = \det(g_{jk})$ and $[g^{jk}]_{j,k=1}^n = g(x)^{-1}$.
Let $M \subset \R^n$ be a compact domain with smooth boundary.
We suppose that $g$ is known only outside of a compact subset $K \subset M^{int}$
and that the source $F$ is a realization of a random variable with the Gaussian white noise distribution
on $(0,\infty) \times \p M$.
Moreover, we assume that the Riemannian manifold $(M,g)$ is non-trapping and 
that $\p M$ is strictly convex with respect to the metric $g$.
We show that the scattering relation of $(M, g)$ is determined by $F$ and the trace of $u$ on $(0,\infty) \times \p M$ almost surely, see Theorem \ref{thm_main} below for the precise formulation.

In particular, if the Riemannian manifold $(M,g)$ is simple, then 
the pair $(F, u)$ on $(0,\infty) \times \p M$ determines $(M,g)$ almost surely
in each of the following cases:
\begin{itemize}
\item[(i)] the dimension $n = 2$ or
\item[(ii)] $n\geq 3$ and the metric is conformally Euclidean, that is, $g_{jk}(x) = a(x) \delta_{jk}$
for a strictly positive function $a$ or
\item[(iii)] $n\geq 3$ and the metric is close to the Euclidean metric.
\end{itemize}
Indeed, by Theorem \ref{thm_main} below, the case (i) follows from \cite{Pestov2005}, 
(ii) follows from \cite{Muhometov1978}, and (iii) from \cite{Burago2010}.
Moreover, there is a conjecture by Gunther Uhlmann
\cite{Uhlmann2004}, that the scattering relation determines 
any non-trapping compact manifold with boundary.
We refer to \cite{Stefanov2009} for work toward resolving the conjecture.

If the source $F$ in (\ref{eq:problem}) can be controlled, that is, 
if we can measure the trace of $u$ on $(0,\infty) \times \p M$ 
for all $F \in C_0^\infty((0,\infty) \times \p M)$,
then the problem to determine $(M,g)$ is equivalent with Gel'fand's inverse problem,
whence it has unique solution \cite{Belishev1987, Belishev1992}.
Contrary to the problem with a single measument as considered in the present paper,
Gel'fand's problem is overdetermined. 
Indeed, the dimension of the data in Gel'fand's problem is $2n-1$
which is strictly greater than the dimension $n \ge 2$ of the unknown $g|_M$.
Notice that $2n-1$ is the number of free variables
of the kernel of the map $F \mapsto u|_{(0,\infty) \times \p M}$,
since $F$ and the trace of $u$ are defined on the $n$ dimensional manifold $(0,\infty) \times \p M$
and the translation invariance in time accounts for the reduction of the dimension by one.
The dimension $n$ of the single-measurement data equals to the dimension of the trace of $u$.

In fact, most of the thoroughly studied inverse boundary value problems are overdetermined. 
Calderon's inverse problem is overdetermined in dimensions $n \geq 3$, see \cite{Sylvester1987}
for the isotropic and \cite{Lee1989, Lassas2001, Lassas2003a, DosSantosFerreira2009}
for the anisotropic case. Likewise, the inverse boundary value problems 
for the wave, heat, and the dynamical Schr\"odinger equations with 
Dirichlet-to-Neumann map as data are all equivalent with the Gel'fand's inverse problem 
\cite{Katchalov2001}, and they are overdetermined in dimensions $n\geq 2$.
However, the two dimensional Calderon's inverse problem, see \cite{Nachman1996, Astala2006, Bukhgeim2008}
for the isotropic and \cite{Sylvester1990, Astala2005, Guillarmou2011} for the anisotropic case, is an 
example of a formally determined inverse problem, that is, the dimension of the data equals to that of the unknown.

In \cite{Helin2010}, we solved a formally determined inverse problem for the wave equation with a single measurement.
Although satisfactory in terms of the dimensions, the result \cite{Helin2010}
relies on the use of a source $F = F_\delta$ given as a weighted sum of point sources, 
\begin{align*}
F_{\delta}(t,x) = \sum_{j=1}^\infty 2^{-2^{j}} \delta_{x_j}(x)\delta(t).
\end{align*}
As the weights vanish superexponentially, the source $F_\delta$ may be hard to realize 
in practice with sufficient precision. On the other hand, random noise sources, as considered in the present paper,
appear in many applications. In seismology, cross-correlations of signal amplitudes generated by ambient seismic noise source are used to study travel times inside the Earth \cite{dehoop_etal_06, GP09}. 
Moreover, in one-dimensional radar imaging models, 
white noise signals are considered to be optimal sources when imaging a
stationary scatterer \cite{Toomay2004}.
Such models correspond mathematically to the one-dimensional deconvolution problem, 
and the present problem can be seen as a multi-dimensional analogue that is translation invariant in only one direction. 

Although we are ultimately concerned of (\ref{eq:problem}) with 
just a single realization of the white noise as the source,
we consider (\ref{eq:problem}) to be a stochastic partial differential equation.
In particular, we show that the pair $(F, u)$ on $(0,\infty) \times \p M$
is a Gaussian random variable and study its ergodicity properties. 
The literature on the stochastic hyperbolic equation and the direct problem is extensive. To our knowledge, the earliest existence and uniqueness results were given in \cite{Carmona88a} for a one-dimensional setting. The research has then extended to higher dimensions and more generalized setting (e.g. \cite{Peszat1, Ond2}), geometrical wave equations \cite{Ond1} and to non-gaussian sources \cite{Nualart}, to name a few directions. 
The pathwise properties have been studied in e.g. \cite{SanzSole1, SanzSole2}. In our problem formulation we have a white noise source which is supported on the boundary of a manifold. Closely related results with boundary supported white noise have been introduced in the work by Dalang and L{\'e}v{\^e}que \cite{DalangLeveque2004, DalangLeveque2006}. For hyperbolic equations with random boundary conditions, see \cite{Peszat2}.

Inverse problems related to stochastic wave equations have been mostly studied in the framework of random media. For this imaging setting we refer to \cite{Fouque_etal, Borcea_etal, BalRyzhik03} and the extensive research by their authors. 
Let us also mention the interesting approaches to stochastic inverse problems taken in \cite{LPS2008, Stuart1, Stuart2}.
 
In addition to our approach, we are aware of two other methods to solve 
formally determined hyperbolic inverse problems.
First, the adaptation of the Gelfand-Levitan method to multidimensional problems,
see \cite{Rakesh2003, Rakesh2011},
assumes that the problem is close to being symmetric in all but one direction.
Second, the Carleman estimates based approach, see \cite{Bukhgeuim1981,Isakov2006,Imanuvilov2003,Bellassoued2008,Stefanov2011},
assumes that the initial data is non-zero and satisfies certain conditions.

\section{Statement of the results}

Let us begin by introducing the scattering relation. 
Below, the tangent space of $M$ is denoted by $TM$ and $\dot \gamma$ denotes 
the tangent vector of a smooth curve $\gamma : [a,b] \to M$. We set
$SM=\{(x,\xi)\in TM;\ |\xi|_g=1\}$ to be the unit sphere bundle on $M$ and write
\begin{equation*}
\p_\pm SM=\{(x,\xi)\in SM;\ x\in \p M,\ \pm(\nu,\xi)>0\} 
\end{equation*}
where $\nu$ is the interior normal vector of $\p M$. 
Further, we denote by $\gamma(t;x,\xi)$ the geodesic with initial data $(x,\xi)\in TM$ and
we set 
\begin{align*}
\tau(x,\xi) = \inf \{ t\in (0, \infty] ; \gamma(t;x,\xi) \notin  M \}.
\end{align*}
We say that $M$ is trapping if the set in the definition of $\tau$ is 
empty for some $(x,\xi) \in TM$.
We make the standing assumption on the metric $g$:
\begin{itemize}
	\item[(A1)] 
	All geodesics $\gamma$ of $(\R^n, g)$ move eventually off to infinity, that is, 
	$|\gamma(t)| \to \infty$ as $t \to \infty$.
\end{itemize}
The assumption (A1) implies, in particular, that $M$ is non-trapping. 
Then the scattering relation of $(M,g)$ is defined by 
\begin{align*}
&\Sigma_{M,g}: \p_+ SM \to (0, \infty) \times \p_- SM,
\\&
\Sigma_{M,g}(x,\xi)=(\tau(x,\xi), \gamma(\tau(x,\xi);x,\xi),\dot \gamma(\tau(x,\xi);x,\xi))
\end{align*}
We write  $\Sigma=\Sigma_{M,g}$ when considering a fixed Riemannian manifold. 

Let us now consider the equation (\ref{eq:problem}) more carefully.
Let $(\Omega, \mathcal F, {\mathbb P})$ be a complete probability space and 
let $W$ be a random variable with the Gaussian white noise distribution 
supported on $\R \times \partial M$.
For a construction of the distribution, see Appendix \ref{sec:appendix}.
To avoid technicalities arising from compatibility conditions between the source term and the vanishing initial conditions we consider the stochastic wave equation 
\begin{align}
\label{wave_eq_with_compatibility_intro}
&\p_t^2 u - \Delta_g u = \chi_+  W \quad \text{on $(0, \infty) \times \R^n$},
\\
&u|_{t = 0} = \p_t u|_{t=0} = 0,\nonumber
\end{align}
where $\chi_+ \in C^\infty(\R)$ is a fixed cut-off satisfying $\chi_+ = 0$ 
and $\chi_+ = 1$ in neighborhoods of $(-\infty,0]$ and $[1,\infty)$, respectively.

Notice that the wave front set $WF(u)$ of the solution of (\ref{wave_eq_with_compatibility_intro})
intersects the conormal bundle of $(0,\infty)\times \p M$ almost surely
since the support of a realization of $W$ coincides with $\R \times \p M$ 
almost surely.
Thus the trace $u|_{(0,\infty) \times \p M}$ can not be defined 
in the sense of distributions. Instead, in Section \ref{subsec:traces}
we define the trace in a scattering sense. 
Let us suppose that $K \subset M^\inter$ is compact and that $g|_{\R^n \setminus K}$ is known.
We choose a smooth Riemannian metric tensor $g_0$ such that $g_0 = g$ in $\R^n \setminus K$
and consider the wave equation
\begin{align*}
&\p_t^2 u_{in} - \Delta_{g_0} u_{in} = \chi_+  W(\omega_0) \quad \text{on $(0, \infty) \times \R^n$},
\\
&u_{in}|_{t = 0} = \p_t u_{in}|_{t=0} = 0\nonumber
\end{align*}
for some fixed $\omega_0\in\Omega$, that is, $W(\omega_0)$ is a realization of the random process $W$.
The measurement operator $L_{M,g}$ is defined by
\begin{align*}
L_{M,g}( W(\omega_0)) = u_{sc}|_{(0,\infty) \times \p M}
\end{align*}
for the scattered wave $u_{sc} = u - u_{in}$.
We abbreviate $L=L_{M,g}$ when considering a fixed Riemannian manifold. 
The Sobolev regularity properties of the random variable $L( W)$
are reviewed in Section \ref{sec:whitenoise}.

To avoid technicalities, we make the following standing assumptions: 
\begin{itemize}
	\item[(A2)] $\p M$ is strictly convex with respect to the metric $g$ and
	\item[(A3)] $g$ coincides with the Euclidean metric outside a compact set. 
\end{itemize}
We believe that (A2) is not an essential assumption. In fact, it is not needed in our previous work 
\cite{Helin2010}.
We are now ready to formulate our main result.



\begin{theorem}
\label{thm_main}
Let $M \subset \R^n$ be a compact domain with smooth boundary. Suppose that two smooth Riemannian metrics $g$ and $\tilde g$ 
satisfy the assumptions (A1)-(A3). 
Furthermore, let $K \subset M^\inter$ be compact
and assume that $g|_{\R^n \setminus K}=\tilde g|_{\R^n \setminus K}$. 
Then for almost every $\omega_0\in\Omega$, the identity
$L_{M,g}( W(\omega_0)) = L_{M, \tilde g}( W(\omega_0))$
implies that $\Sigma_{M,g} = \Sigma_{M, \tilde g}$.
\end{theorem}

\subsection{Outline of the proof} Let us summarize how Theorem \ref{thm_main} is obtained. First, for an open set $B\subset \R^{k}$ denote the dual pairing of a generalized function $f\in \D'(B)$ and a test function $g \in C^\infty_0(B)$ by $(f,g)_{\D'\times C^\infty_0(B)}$.
Using integration by parts we show in Section \ref{subsec:integration_by_parts}
that for almost every $\omega_0$, the pairing
$(W(\omega_0), \chi_+ w)_{\D' \times C_0^\infty((0,T) \times \p M)}$
is determined by the data pair $(W(\omega_0), L W(\omega_0))$ for fixed $T > 0$  
and for such a smooth solution $w$ of $(\p_t^2 - \Delta_g) w = 0$ 
that the state $(w(T), \p_t w(T))$ is supported outside $M$.

We choose $w$ to be a Gaussian beam solution that is sent backwards in time from the exterior of $M$.
Section \ref{subsec:gaussianbeams_intro} is devoted to 
the introduction of Gaussian beams.
A backward Gaussian beam is concentrated on a geodesic $\gamma$ and is essentially determined by 
the end point $\gamma(T)$ and direction $\dot \gamma(T)$ and a scaling parameter $\epsilon > 0$.
We assume that a backward Gaussian beam $w_\epsilon$ enters in $M$ at time $T-r$ and write
$(x,\xi) = (\gamma(T-r),-\dot \gamma(T-r))$. 
In Section \ref{subsec:dual_pairing} we construct an oscillating test function $\psi_\epsilon$ that imitates
a Gaussian beam passing through a point $(y,\eta) \in \p_+ SM$ at time $s > 1$. 
The crux of the method lies in Theorem \ref{lem_pairing_asymptotics}, where we obtain the following asymptotics
\begin{equation}
	\label{intro:indicator}
\lim_{\epsilon\to 0}\; \epsilon^{-\frac n2}(\psi_\epsilon, w_\epsilon)_{L^2((0,T) \times \p M)}
\begin{cases}
\neq 0, & \textrm{if } (T-r-s,y,\eta) = \Sigma(x,\xi),\\
= 0, & {\rm otherwise.}
\end{cases}
\end{equation}

In Section \ref{sec:rvs_and_asymptotics} we consider the correlations of two random variables
$X_\epsilon = (W, \epsilon^{-n/4} \psi_\epsilon)$
and
$Y_\epsilon = (W, \epsilon^{-n/4} \chi_+ w_{\epsilon})$, and in Section \ref{sec:ergodicity}
we combine energy decay and ergodicity arguments to show that
\begin{equation}
	\lim_{N \to \infty} \frac{1}{N^3} \sum_{j=1}^{N^3} X_\epsilon^j  Y^j_\epsilon = \expec X_\epsilon Y_\epsilon = 
	\epsilon^{-\frac n2}(\psi_\epsilon, w_\epsilon)_{L^2((0,T) \times \p M)}
\end{equation}
for time-translated variables $X^j_\epsilon$ and $Y^j_\epsilon$.
Since the variables $X^j_\epsilon$ and $Y_\epsilon^j$ are determined by $(W,L W)$ almost surely, we find out if $w_\epsilon$ and $\phi_\epsilon$ coincide in the sense of equation \eqref{intro:indicator}.
By repeating the argument for a dense numerable set of initial data, we obtain the scattering relation by continuity results.

\section{The stochastic direct problem}
\label{sec:whitenoise}

\subsection{White noise and generalized solutions}

We recall that a random variable with the Gaussian white noise distribution 
supported on $\R \times \p M$ can be defined in the local Sobolev spaces,
\begin{align*}
W : \Omega \to H_{loc}^{-(n+1)/2-\epsilon}(\R^{1+n}),
\end{align*}
where $\epsilon > 0$, see Appendix \ref{sec:appendix} for more details. 
The characterizing property of the white noise on $\R \times \p M$
is the following 
equation that holds for any $\phi, \psi \in C_0^\infty(\R^{1+n})$ 
\begin{equation}
\label{the_correlations}
\expec (W, \phi)_{\D' \times C^\infty_0(\R^{1+n})}(W, \psi)_{\D' \times C^\infty_0(\R^{1+n})}
= (\phi, \psi)_{L^2(\R \times \partial M)}.
\end{equation}
Here and throughout the paper we are using real inner products.

Let us recall that the equation
\begin{align}
\label{wave_eq_with_compatibility}
&\p_t^2 u - \Delta_g u = \chi_+ f \quad \text{on $(0, \infty) \times \R^n$},
\\
&u|_{t = 0} = \p_t u|_{t=0} = 0,\nonumber
\end{align}
has a unique solution 
$u \in \D'((0, \infty) \times \R^n)$ for all $f \in  \D'((0, \infty) \times \R^n)$,
see e.g. \cite[Lem. 5.1.5]{Duistermaat1996} for uniqueness.
The existence follows by transposing the smooth case, see e.g. \cite[Th. 7.2.7]{Evans1998}.
Moreover, the parametrix construction \cite[Th. 26.1.14]{Hormander1985a} implies that 
the solution map ${\mathcal K}_g: f \mapsto u$ 
is continuous
\begin{equation}
\label{first_continuity_of_sol_operator}
{\mathcal K}_g : H_{loc}^{s}((0,\infty)\times \R^n) \to H_{loc}^{s+1}((0,\infty)\times \R^n),
\quad s \in \R.
\end{equation} 

Due to \cite[Prop. 3.7.2]{Bogachev} we have that
\begin{equation}
	U = {\mathcal K}_g W : \Omega \to H_{loc}^{-(n-1)/2-\epsilon}((0,\infty) \times \R^n)
\end{equation}
is a well-defined Gaussian random variable.
Moreover, $U$ satisfies
\begin{align}
\label{wave_eq_with_randomness}
&\p_t^2 U - \Delta_g U = \chi_+  W \quad \text{on $(0, \infty) \times \R^n$},
\\
&U|_{t = 0} = \p_t U|_{t=0} = 0.\nonumber
\end{align}
almost surely in the sense of distributions. 

We point out that the solution $U$ can be shown to have stronger pathwise properties, 
see e.g. \cite{DalangLeveque2004}. 
However, such results are not crucial in this treatise, whereas the random variable formalism provides us some flexibility for the analysis of the inverse problem. 

\subsection{Trace of the solution in the scattering sense}
\label{subsec:traces}

\def\Tr{{\rm Tr}}
In below, we will use regularity properties of the traces 
\begin{align*}
\Tr_{\p M} {\mathcal K}_g f = u|_{(0, \infty) \times \p M},
\quad \text{and} \quad \Tr_\Omega^T {\mathcal K}_g f = u|_{\{T\} \times \Omega},
\end{align*}
where $T > 0$ and $\Omega \subset \R^n$ is open.
The above traces are not well defined in the sense of distributions 
if the wavefront set $WF(f)$ intersects the conormal bundles of 
the sets 
\begin{align}
\label{trace_surfaces}
(0,\infty)\times \p M, \quad \{T\} \times \Omega.
\end{align}
However, the traces are well defined by \cite[Th. 26.1.14]{Hormander1985a}
if $\supp(f)$ does not intersect the sets (\ref{trace_surfaces}).

Let $B \subset \R^n$ be open and let us define the closed subspace  
\begin{align*}
H_{locc}^s((0,\infty)\times B) \subset H_{loc}^{s}((0,\infty)\times \R^n)
\end{align*}
consisting of distributions supported in $[0,\infty] \times \overline B$.
If $\overline B \cap \overline \Omega = \emptyset$, then we have the regularity
\begin{align}
\label{trace_T}
\Tr_\Omega^T {\mathcal K}_g : 
H_{locc}^s((0,\infty)\times B) \to H_{loc}^{s+1}(\Omega),
\quad s \in \R.
\end{align}
By \cite{Hormander1971} the mapping $\Tr_\Omega^T {\mathcal K}_g$
is a Fourier integral operator of order $-5/4$ with the canonical relation $C$
consisting of the points 
\begin{align*}
&(t, x, |\xi|, \xi, \gamma(T-t; x, \hat \xi), \dot \gamma(T-t; x, \hat \xi)),
\quad (x,\xi) \in T^* B \setminus 0,\ t > 0,
\end{align*}
such that $\gamma(T-t; x, \hat \xi) \in \Omega$.
Here $\hat \xi = \xi / |\xi|$ 
and we have identified the cotangent and the tangent space using the metric $g$.
In particular, the canonical relation is parametrized by 
$(t, x, \xi)$ and the projection $C \to T^* (0,\infty)\times B$ has the rank $2(n+1) - 1$.
We may apply \cite[Th. 4.3.2]{Hormander1971} to get the continuity (\ref{trace_T}).

Analogously, if $\overline B \cap \p M = \emptyset$
and $\p M$ is strictly convex with respect to the metric $g$, then we have the regularity
\begin{align}
\label{trace_boundary}
\Tr_{\p M} {\mathcal K}_g : 
H_{locc}^s((0,\infty)\times B) \to H_{loc}^{s+1}((0,\infty)\times \p M),
\quad s \in \R.
\end{align}
Indeed, $\Tr_{\p M} {\mathcal K}_g$
is a Fourier integral operator of order $-5/4$ with the canonical relation 
consisting of the points 
\begin{align*}
&(t, x, |\xi|, \xi, T, \gamma(T-t; x, \hat \xi), |\xi|, \dot\gamma^\t(T-t; x, \hat \xi)),
\end{align*}
such that $\gamma(T-t; x, \hat \xi) \in \p M$, $(x,\xi) \in T^* B \setminus 0$
and $t, T > 0$.
Here $v \mapsto v^\t$ is the projection $T^* \R^n \to T^* \p M$.
The strict convexity of $\p M$ implies that $T$ is locally a function of 
$(t, x, \hat \xi)$. As above, the continuity (\ref{trace_boundary}) follows from \cite[Th. 4.3.2]{Hormander1971}.
The convexity assumption is not essential here
since we could use the result \cite{Greenleaf1994} as in \cite{Tataru1998}.
However, we will use the strict convexity assumption also for other purposes in below.


Let $K \subset M^\inter$ be compact and let $g_0$ be a smooth Riemannian metric tensor 
such that $g_0 = g$ in $\R^n \setminus K$.
We consider the incoming wave defined by $u_{in} = {\mathcal K}_{g_0} f$. 
Then the scattered wave $u_{sc} = u - u_{in}$ satisfies
\begin{equation}
\label{scattering_sol}
u_{sc} = ({\mathcal K}_g 
- {\mathcal K}_{g_0})  f =
{\mathcal K}_g \tilde f
\end{equation}
where $\tilde f = (\Delta_{g} - \Delta_{g_0})
u_{in}$ is supported in $(0, \infty) \times K$. 
Thus the measurement operator 
\begin{align}
\label{measurement_op}
Lf = {\rm Tr}_{\p M} {\mathcal K}_g 
(\Delta_{g} - \Delta_{g_0}) {\mathcal K}_{g_0} f =
u_{sc}|_{(0,\infty) \times \p M}
\end{align}
is continuous 
\begin{align*}
L : H_{locc}^s((0,\infty)\times B) \to H_{loc}^{s}((0,\infty)\times \p M),
\end{align*}
where $B \subset \R^n$ is a neighborhood of $\p M$ satisfying $K \cap \overline B = \emptyset$. 

From these considerations it follows that 
we can factorize 
\begin{equation}
\label{eq:Ufact}
U = U_{in} + U_{sc}, \quad \textrm{almost surely,} 
\end{equation}
where
$U_{in} = {\mathcal K}_{g_0}  W$ and $U_{sc} = {\mathcal K}_g
(\Delta_{g} - \Delta_{g_0}) {\mathcal K}_{g_0}  W$.
Moreover, the measurement 
\begin{equation*}
	L W : \Omega \to H_{loc}^{-(n+1)/2-\epsilon}((0,\infty)\times \p M)
\end{equation*} 
is a well-defined Gaussian random variable.

\subsection{Integration by parts}
\label{subsec:integration_by_parts}

Following \cite{Helin2010} we
have for a smooth source $f \in C_0^\infty((0, \infty) \times \R^n)$ and $T > 0$ that
\begin{align}
\label{integration_by_parts}
&(\chi_+ f, w)_{L^2((0, T) \times M)} 
\\\nonumber&\quad= 
(\p_t u(T), w(T))_{L^2(\R^n)} - (u(T), \p_t w(T))_{L^2(\R^n)},
\end{align}
where $u$ is the solution of (\ref{wave_eq_with_compatibility})
and $w \in C^\infty([0,T] \times \R^n)$
satisfies the wave equation $(\p_t^2 - \Delta_g) w = 0$. 
We let 
\begin{align}
\label{support_of_test_function} 
w(T), \p_t w(T) \in C_0^\infty(\R^n \setminus M) 
\end{align}
and denote $\Omega = \supp(w(T)) \cup \supp(\p_t w(T))$.
Let $B \subset \R^n$ be a neighborhood of $\p M$ such that $\overline B \cap \Omega = \emptyset$.
The density of the embedding 
\begin{align*}
C_0^\infty((0, \infty) \times B) \subset H_{locc}^s((0,\infty)\times B)
\end{align*}
and the continuity (\ref{trace_T}) imply that the identity
\begin{align}
\label{random_integration_by_parts}
&(\chi_+ W, w)_{\D' \times C_0^\infty((0, T) \times \p M)} 
\\\nonumber&\quad= 
(\p_t U(T), w(T))_{\D' \times C_0^\infty(\R^n)} - (U(T), \p_t w(T))_{\D' \times C_0^\infty(\R^n)},
\end{align}
holds almost surely.


Now, we insert \eqref{eq:Ufact} to the right-hand side of (\ref{integration_by_parts}) whence it splits into four terms.
The two terms corresponding to $U_{in}$ are almost surely determined by $W$ and $g_0$.
We show next that the traces ${\rm Tr}^T_{\Omega} \p^j_t U_{sc}$, $j=0,1$
are almost surely determined from the solutions of the exterior problem
\begin{align}
\label{exterior_problem}
&\p_t^2 v - \Delta_g v = 0, \quad \text{in $(0, T)\times \R^n \setminus M$,}
\\\nonumber
&v|_{(0,T) \times \p M} = h,
\\\nonumber
&v|_{t = 0} = 0, \quad \p_t v|_{t = 0} = 0.
\end{align}

Let us write ${\mathcal K}_{ex} : h \mapsto v$ 
and recall, see e.g. \cite{Helin2010}, that the maps 
${\rm Tr}^T_{\Omega} \p^j {\mathcal K}_{ex}$, $j=0,1$, are continuous
\begin{align*}
\E'((0,T) \times \p M) \to \D'(\Omega).
\end{align*}
Notice that $U_{sc}|_{(0,\infty) \times \p M}$ vanishes almost surely 
for small $t > 0$.
Moreover, we may introduce a cut-off function $\chi_T \in C^\infty(0,T)$
satisfying $\chi_T = 0$ near $t=T$ and $\chi_T = 1$ away from a neighborhood of $T$.
By choosing the neighborhood small enough, we have by finite speed of propagation that 
${\rm Tr}^T_{\Omega} \p^j_t {\mathcal K}_{ex} h = {\rm Tr}^T_{\Omega} \p^j_t {\mathcal K}_{ex} (\chi_T h)$, $j = 0,1$, 
since the distance between $\Omega$ and $\p M$ is strictly  positive.
Hence the traces of the scattered waves are determined as
\begin{equation}
	{\rm Tr}^T_{\Omega} \p^j_t U_{sc} =
	{\rm Tr}^T_{\Omega} \p^j_t {\mathcal K}_{ex} (\chi_T L W),
	\quad j=0,1,
\end{equation}
almost surely.

In particular, a realization of the measurement $L W$
almost surely determines the realization of the distribution pairing via
\begin{multline}
\label{pairing_Ww}
(W, \chi_+ w)_{\D' \times C_0^\infty((0, T) \times \p M)} \\
=  (\p_t U_{in}(T) + {\rm Tr}^T_{\Omega} \p_t {\mathcal K}_{ex} (\chi_T L W), w(T))_{\D' \times C_0^\infty(\R^n)} 
\\
- (U_{in}(T) + {\rm Tr}^T_{\Omega} {\mathcal K}_{ex} (\chi_T L W), \p_t w(T))_{\D' \times C_0^\infty(\R^n)},
\end{multline}
where $w$ satisfies the wave equation $(\p_t^2 - \Delta_g) w = 0$
with the initial conditions as in (\ref{support_of_test_function}).

\section{Gaussian beams}
\label{sec:gaussianbeams}

\subsection{Gaussian beam solutions}
\label{subsec:gaussianbeams_intro}
In what follows we will choose the test function $w$ in (\ref{pairing_Ww})
to be a Gaussian beam solution to the wave equation $(\p_t^2 - \Delta_g) w = 0$.
A formal Gaussian beam of order $N_U \in \N$ propagating on a geodesic $\gamma_U$ is a function of the form 
\begin{align*}
U_{\epsilon}(t,x) = 
e^{i\epsilon^{-1}\theta(t,x)}
\sum _{n=0}^{N_U} \epsilon^n u_n(t,x), \quad t\in \R,\ x \in \R^n.
\end{align*}
Here the phase function $\theta$ is of the form 
\begin{align}
\label{def_theta}
\theta(t, x) &= p(t) (x - \gamma_U(t)) + \frac{1}{2} (x - \gamma_U(t)) H(t) (x - \gamma_U(t)) 
\\\nonumber&\quad\quad+ \sum_{2 < |\alpha| \le N} \frac{\theta_\alpha(t)}{\alpha !} (x - \gamma_U(t))^\alpha,
\end{align}
where all the coefficients are smooth functions and, in particular, 
$p$ is the covariant representation of the velocity $\dot \gamma_U$, 
that is, 
$p = \dot \gamma_U^\flat$, 
and $H$ is a symmetric matrix with positive definite imaginary part.
Moreover, the principal amplitude function $u_0$ satisfies $u_0(t, \gamma_U(t)) \ne 0$.

Let $T > 0$. Then the phase and the amplitude functions are constructed so that there is a neighborhood $V_U$ of the trajectory
\begin{align*}
\{ (t, \gamma_U(t));\ t \in [0, T] \} \subset \R \times \R^n
\end{align*}
and $C > 0$ such that  
\begin{align*}
&\norm{(\p^2_t - \Delta_g) U_\epsilon}_{C^k(\bar{V_U})} \le C \epsilon^{N_U-k}, \quad k < N_U.
\end{align*}
See \cite[Cor. 2.63]{Katchalov2001} for a proof of this estimate and \cite{Babich1972, Babich1985, Babich1981, Ralston1982} for earlier references of Gaussian beams. Moreover, we may choose $V_U$ so that there is $\beta_\theta > 0$ satisfying
\begin{align}
\label{im_theta_bound}
&\Im \theta(t, x) \ge \beta_\theta d^2(x, \gamma(t)), \quad (t, x) \in \bar{V_U}.
\end{align}

Notice that $H(t)$ is non-degenerate since 
\begin{align}
\label{assumption_H}
\Im H(t) = \frac{H(t) - \bar{H(t)}}{2i} > 0.
\end{align}
Indeed, by symmetry of $H(t)$ we have for nonzero $\xi \in \C^n$ that
\begin{align*}
\Im (\xi \bar{H(t) \xi} )
&= 
\frac{1}{2i} (\xi \bar{H(t) \xi}  - \bar{\xi} H(t) \xi)
= 
\frac{1}{2i} (\xi \bar{H(t) \xi}  - \xi H(t) \bar{\xi})
\\&= 
\xi(\Im H(t)) \bar{\xi}   > 0,
\end{align*}
whence $H(t) \xi \ne 0$.

A Gaussian beam solution is given by the following theorem, see \cite[Th. 2.64]{Katchalov2001} 
for a proof. 

\begin{theorem}
\label{th_from_gb_to_sol}
Let $T > 0$, $\chi_U \in C_0^\infty(V_U)$ and 
consider the wave equation 
\begin{align*}
&\p_t^2 w_\epsilon - \Delta_g w_\epsilon = 0 \quad \text{on $(-\infty, T) \times \R^n$},
\\&w_\epsilon(T) = (\chi_U U_\epsilon)(T),\ 
\p_t w_\epsilon(T) = \p_t(\chi_U U_\epsilon)(T).
\end{align*}
Then there is $C > 0$ such that
\begin{align*}
\norm{w_\epsilon - \chi_U U_\epsilon}_{C^k([0, T] \times \R^n)} \le C \epsilon^{N_U-k}, \quad k < N_U.
\end{align*}
\end{theorem}

We will also need to estimate $w_\epsilon|_{(-\infty, -S) \times \p M}$ for $S > 0$.
This is obtained from an energy decay estimate by Va{\u\i}nberg \cite{Vauinberg1989}
that follows from the assumptions (A1) and (A3).

\begin{theorem}
	\label{th_energy_decay}
	Let $\hat M \subset \R^n$ be a compact set.	
	There is $T = T_{ED} > 0$ such that if $w$ is a smooth solution to the wave equation 
	\begin{align*}
		&\p_t^2 w - \Delta_g w = 0 \quad \text{on $(-\infty, T) \times \R^n$},
	\end{align*}
	satisfying ${\rm supp} (w(T)) \cup {\rm supp} (\p_t w(T)) \subset \hat M$, then for all $S > 0$
	\begin{equation}
		\left\|w\right\|^2_{L^2((-\infty,-S) \times \p M)} \leq C \eta(S)\left(\norm{w(T)}_{L^2(\hat M)}^2 + \norm{\p_t w(T)}_{L^2(\hat M)}^2\right),
	\end{equation}
	where
	\begin{equation}
		\eta(S) = 
		\begin{cases}
			S^{-2n+3} & \textrm{for even } n, \\
			\exp(-\delta S) & \textrm{for odd } n.
		\end{cases}
	\end{equation}
	Above, $C, \delta > 0$ are constants which depend only on $\hat M$ and on the Riemannian manifold $(\R^n,g)$.
\end{theorem}
\begin{proof}
According to \cite[Th. 2.104]{EgorovShubin} we have 
for some $T > 0$ and $C > 0$ the estimate
	\begin{equation}
		\left| w(x, -t)\right| \leq C \eta'(t)\left(\norm{w(T)}_{L^2(\hat M)} + \norm{\p_t w(T)}_{L^2(\hat M)}\right),
	\end{equation}
	where $x \in \hat M$, $t>0$ and
	\begin{equation}
		\eta'(t) = 
		\begin{cases}
			t^{-n+1} & \textrm{for even } n, \\
			\exp(-\delta t) & \textrm{for odd } n.
		\end{cases}
	\end{equation}
The claim follows after integration.
\end{proof}

\subsection{Dual pairing of beam-like functions}
\label{subsec:dual_pairing}

In the following we study dual pairing of Gaussian beams with oscillating functions having a phase of the form
\begin{align}
\label{eq:beamlikephase}
\theta(t, x) &= p(t) (x - \gamma(t)) + \frac{1}{2} (x - \gamma(t)) H(t) (x - \gamma(t)) 
\\\nonumber&\quad\quad+ \O(|x - \gamma(t)|^3).
\end{align}
where $\gamma$ and $p$ are smooth paths in $\R^n$ and $H \in \C^{n \times n}$ is symmetric.
Moreover, we assume that there is $t_0 \in (0, T)$ such that
\begin{itemize}
\item[(B1)] $\gamma(t_0) \in \p M$ and
$\dot \gamma(t_0)  \notin T \p M$,
\item[(B2)] $p(t_0) = \dot \gamma(t_0)^\flat$ and
\item[(B3)] $\Im H(t_0)$ is positive definite.
\end{itemize}
We employ the method of stationary phase to get the following lemma. 

\begin{lemma}
\label{lem_stationary_phase}
Let $\theta_1$ and $\theta_2$ be two phase functions of the form \eqref{eq:beamlikephase} such that for some $t_0 \in (0,T)$ both phases satisfy  assumptions (B1)-(B3). We use notation $\gamma_j$, $p_j$ and $H_j$ for the paths and the matrix for the phase $\theta_j$, $j=1,2,$ respectively. Suppose the paths $\gamma_j$ collide at $t_0$, that is, 
\begin{equation}
\gamma_1(t_0) = \gamma_2(t_0) \quad {\rm and} \quad
\dot \gamma_1(t_0) = \dot \gamma_2(t_0).
\end{equation}
It follows that there exists an open
neighbourhood $V$ of $(t_0,x_0)$ such that if $u \in C_0^\infty(\R \times \R^n)$ and $\supp u \subset V$ then 
\begin{align*}
\int_0^\infty \int_{\p M} e^{i\epsilon^{-1} \theta_1(t, x)} \bar{e^{i\epsilon^{-1} \theta_2(t, x)}} u\ dS dt
= \beta \epsilon^{n/2} + \O(\epsilon^{n/2 + 1}),
\end{align*}
where $\beta \ne 0$ if and only if $u(t_0, x_0) \ne 0$. 
\end{lemma}
\begin{proof}
Let us choose such coordinates in a neighborhood of $x_0$ 
that $\p M = \{x \in \R^n;\ x^n = 0 \}$.
We write $x = (x^\t, x^n) \in \R^n$ where $x^\t = (x^1, \dots, x^{n-1}) \in \R^{n-1}$. Then we have for a phase function of the form \eqref{eq:beamlikephase} that
\begin{align*}
\p_{x^\t} \theta &= p(t)^\t + \ll( H(t)(x - \gamma(t)) \rr)^\t + \O \ll(|x - \gamma(t)|^2\rr),
\\
\p_{t} \theta &= - p(t) \dot \gamma(t) + \p_t p(t) (x - \gamma(t)) - \dot \gamma(t) H(t) (x - \gamma(t)) 
\\&\quad\quad+ \O \ll(|x - \gamma(t)|^2\rr),
\\
\p_{x^\t}^2 \theta &= H(t)^\t + \O \ll(|x - \gamma(t)|\rr),
\\
\p_t \p_{x^\t} \theta &= \p_t p(t)^\t - \ll( H(t)\dot \gamma(t) \rr)^\t + \O \ll(|x - \gamma(t)|\rr),
\\
\p_{t}^2 \theta &= -2 \p_t p(t) \dot \gamma(t) - p(t) \p_t \dot \gamma(t) + \dot \gamma(t) H(t) \dot \gamma(t)
\\&\quad\quad+ \O \ll(|x - \gamma(t)|\rr),
\end{align*}
where $H(t)^\t = (H(t))_{j,k=1}^{n-1}$.

Let us consider the map $\Theta(t, x^\t) = (\theta_1 - \bar{\theta_2})(t, x^\t, 0)$.
Assumptions (B1) and (B2) imply that $(t_0, x_0^\t)$ is a critical point of $\Theta$.
We will show next that the critical point is non-degenerate.
As the Hessian  
$(\p_{t, x^\t}^2 \Theta)(t_0, x_0^\t)$ is symmetric, it is enough to show that 
its imaginary part is positive definite. 
We denote $\xi_0 = \dot \gamma_1(t_0)$ and 
\begin{align*}
H_0 = \Im H_1(t_0) + \Im H_2(t_0).
\end{align*}
Let $(\alpha, v) \in \R \times \R^{n-1}$ be nonzero. 
We denote $w = (v, 0) \in \R^n$. Notice that $\xi_0^n \ne 0$ since $\xi_0 \notin T \p M$.
Thus $w - \alpha \xi_0 \ne 0$ and as $H_0$ is positive definite by assumption (3), we have
\begin{align*}
0 &< (w - \alpha \xi_0) H_0 (w - \alpha \xi_0)
\\&= 
v (H_0)^\t v - 2 \alpha \ll( H_0 \xi_0 \rr)^\t v + \alpha^2 \xi_0 H_0 \xi_0
\\&= 
(\alpha, v) (\Im \p_{t, x^\t}^2 \Theta(t_0, x_0^\t)) (\alpha, v).
\end{align*}

As $(t_0, x_0^\t)$ is non-degenerate, there are no other critical points in $\supp(u)$ 
if it is small enough.
Moreover, for small $\supp(u)$, we have that $\Im \Theta \ge 0$ in $\supp(u)$ 
since the first order term in $x - \gamma_j(t)$ is real 
and the imaginary part of the second order term is non-negative and dominates the higher order terms. 

The method of stationary phase \cite[Th.7.7.5]{Hormander1990} gives
\begin{align*}
&\ll| \int_0^\infty \int_{\p M} e^{i\epsilon^{-1} \Theta(t, x)} u\ dS dt - \epsilon^{n/2} \ll( \det(\p_{t, x^\t}^2 \Theta / 2\pi i)^{-1/2} u \sqrt{g} \rr)|_{t=t_0, x = x_0} \rr|
\\&\quad\le C \epsilon^{n/2+1} \norm{u}_{C^{2s}},
\end{align*}
where $dS = \sqrt{g} dx^\t$ is the Riemannian volume measure of $\p M$
and $s$ is the smallest integer satisfying $s \ge n/2 + 1$.
\end{proof}

Let us consider the test functions of the form 
\begin{align}
\label{def_psi}
\psi_\epsilon(t, x; s, y, \eta) &= 
\overline{e^{i\epsilon^{-1} \tilde \theta(t, x) }} \chi_\psi(t, x),
\\\nonumber
\tilde \theta(t, x) &= q (x-\mu(t)) + i |x - \mu(t)|^2,
\\\nonumber
\mu(t) &= (t - s) \eta + y,
\end{align}
where $y \in \p M$, $\eta \in \p_+ S_{y} M$, $s > 0$,
$\chi_\psi \in C_0^\infty(\R^{1+n})$ and $q = \eta^\flat$.
We assume that $\chi_\psi(s, y) = 1$ and that the support of $\chi_\psi$ is small
enough so that $\mu(t)$ intersects $\p M$ only at $t=s$ for $t$ such that $(t,x) \in \supp(\chi_\psi)$ for some $x \in \R^n$.

\begin{theorem}
\label{lem_pairing_asymptotics}
Let $T > 0$ and suppose that $w_\epsilon$ is a Gaussian beam solution as in Theorem \ref{th_from_gb_to_sol},
and that the corresponding formal Gaussian beam $U_\epsilon$ has the order $N_U \ge n/4 +1$
and propagates along the geodesic $\gamma_U$.
Suppose, furthermore, that $\gamma_U$ is transversal to $\p M$.
Then 
\begin{align*}
\norm{w_\epsilon}_{L^2((0,T) \times \p M)}^2 = \O(\epsilon^{n/2}).
\end{align*}
Furthermore, let $\psi_\epsilon$ be as in (\ref{def_psi}).
Then
\begin{align*}
\norm{\psi_\epsilon}_{L^2((0,T) \times \p M)}^2 &= \O(\epsilon^{n/2}),
\\
(\psi_\epsilon, w_\epsilon)_{L^2((0,T) \times \p M)} &= \beta\epsilon^{n/2} + \O(\epsilon^{n/2+1}),
\end{align*}
where 
\begin{align}
\label{beta_test}
\beta \begin{cases}
\ne 0, &\text{if $\gamma_U(s) = y$ and $p(s) = q$},
\\
= 0, &\text{if $\gamma_U(s) \ne y$ or $p(s)^\t \ne q^\t$}.
\end{cases}
\end{align}
\end{theorem}
\begin{proof}
Notice that 
\begin{align*}
\norm{w_\epsilon - \chi_U U_\epsilon}_{L^2((0,T) \times \p M)} = \O(\epsilon^{n/4+1}),
\end{align*}
whence it is enough to prove the claimed estimates with $w_\epsilon$
replaced by $\chi_U U_\epsilon$.
Let us write $\gamma = \gamma_U$. As $\gamma$ is transversal to $\p M$, 
there are finitely many points in $\gamma((0,T)) \cap \p M$, say $x_1 = \gamma(t_1), \dots, x_J = \gamma(t_J)$.
If $V_j \subset (0,T) \times \p M$ is a sufficiently small neighborhood of $(t_j,x_j)$, 
then Lemma \ref{lem_stationary_phase} yields that 
\begin{align*}
\norm{\chi_U U_\epsilon}_{L^2(V_j)}^2 = \O(\epsilon^{n/2}).
\end{align*}
Let us write $R = ((0,T) \times \p M) \setminus \bigcup_{j=1}^J V_j$. By (\ref{im_theta_bound}) and \cite[Th. 7.7.1]{Hormander1990} we have that 
\begin{align*}
\norm{\chi_U U_\epsilon}_{L^2(R)}^2 = \O(\epsilon^N),
\end{align*}
for any $N \in \N$. 

Let us now consider the asymptotic estimates involving $\psi_\epsilon$.
By the definition of $\psi_\epsilon$ we have 
\begin{align*}
\supp(\psi_\epsilon) \cap ((0,T) \cap \p M) \cap \{(t, \mu(t)); t \in (0,T)\}= \{(s,y)\}.
\end{align*}
Let $V \subset (0,T) \times \p M$ be a neighborhood of $(s,y)$
and define the set $R = ((0,T) \cap \p M) \setminus V$.
Then $\Im (\theta + \tilde \theta) \ge \Im \tilde \theta > 0$
in $R$, whence
\begin{align*}
\norm{\psi_\epsilon}_{L^2(R)}^2 = \O(\epsilon^{N}),
\quad
(\psi_\epsilon, w_\epsilon)_{L^2(R)} = \O(\epsilon^{N}),
\end{align*}
for any $N \in \N$. Thus it is enough to consider 
the norm and the inner product on $V$.

Lemma \ref{lem_stationary_phase} gives immediately the claimed asymptotic behaviour of the norm of $\psi_\epsilon$ and also that of the inner product if $\gamma(s) = y$ and $p(s) = q$. Under this assumption $\beta \ne 0$ since the principal amplitude function of the Gaussian beam satisfies $u_0(s, \gamma(s)) \ne 0$.

Let us consider the case $\gamma(s) \ne y$. 
If $V$ is small enough, then we have that 
$\Im (\theta + \tilde \theta) \ge \Im \theta> 0$ in $V$ and 
\begin{align*}
(\psi_\epsilon, \chi_U U_\epsilon)_{L^2(V)} = \O(\epsilon^N),
\end{align*}
for any $N \in \N$. 

Let us consider the case $\gamma(s) = y$ and $p(s)^\t \ne q^\t$.
Then 
\begin{align*}
\p_{x^\t} \Re (\theta - \tilde \theta)(s, y) = p(s)^\t - q^\t \ne 0.
\end{align*}
If $V$ is small enough, then 
the function $\Re (\theta - \tilde \theta)$ has no critical points in $V$ and 
\cite[Th. 7.7.1]{Hormander1990} implies 
\begin{align*}
(\psi_\epsilon, \chi_U U_\epsilon)_{L^2(V)} = \O(\epsilon^N),
\end{align*}
for any $N \in \N$. 
\end{proof}

\section{Correlations generated by the white noise source}
\label{sec:rvs_and_asymptotics}

In the following two sections we extract the travel time data by studying the correlation of certain white noise related random processes on the boundary $\partial M$.
%

\subsection{Definition of two random processes}
\label{sec:def_process}

Let us choose a compact set $\hat M \subset \R^n$
such that $M \subset \hat M^\inter$ and a constant
$T > T_{ED}$ such that $T - 1$ 
is greater than the maximum length of a geodesic of $(\hat M, g)$.
Here the constant $T_{ED}$ is as in Theorem \ref{th_energy_decay}.

Let us introduce the notation $\pt$ and $\pgb$ for the points 
\begin{align}
	\label{eq:def_pt}
	\pt &= (s,y,\eta), \quad  (s,y,\eta) \in (1,T) \times  \partial_+ SM 
\\
	\label{eq:def_pgb}
	\pgb &= (T, z, \zeta), \quad (z, \zeta) \in S (\hat M \setminus M). 
\end{align}
In what follows, we are only interested in the points $\pgb$ such that the geodesic $\gamma$ from $(z, \zeta)$ enters in $M$.

We define our test process
\begin{equation}
\label{eq:deftestprocess}
X_\epsilon(\pt) = (W, \epsilon^{-n/4} \psi_\epsilon(\pt))_{\D' \times C_0^\infty((0, T) \times \p M)}.
\end{equation}
where $\psi_\epsilon(t, x; \pt)$ is the test function defined in (\ref{def_psi}).
We assume that the cut-off function $\chi_\psi$ in (\ref{def_psi}) is chosen so that
\begin{align*}
\supp(\psi_\epsilon(\pt)) \subset (1,T) \times \R^n.
\end{align*}
Moreover, we denote by $w_\epsilon(t, x; \pgb)$ the Gaussian beam solution of 
Theorem \ref{lem_pairing_asymptotics}, where we have written out explicitly the dependence on the parameters $z = \gamma_U(T)$ and $\zeta = -\dot \gamma_U(T)$. 
Let us write $\gamma(t) = \gamma_U(T-t)$.
The assumption (A2) implies that if the geodesic $\gamma$ 
enters in $M$ then it is transversal to $\p M$.
Let us define the process 
\begin{equation}
	\label{eq:defgbprocess}
	Y_\epsilon(\pgb) = (W, \epsilon^{-n/4} \chi_+ w_{\epsilon}(\pgb))_{\D' \times C_0^\infty((0, T) \times \p M)}.
\end{equation}
We remind the reader that the measurement $L( W(\omega_0))$
determines the realization $Y_\epsilon(\pgb; \omega_0)$ almost surely via equation \eqref{pairing_Ww}.

Notice that the functions $\pt \mapsto X_\epsilon(\pt)$ and $\pgb \mapsto Y_\epsilon(\pgb)$
are complex-valued zero-centered Gaussian random processes parametrized by sets \eqref{eq:def_pt} and \eqref{eq:def_pgb}, respectively.
Let us define
\begin{align*}
\beta_\epsilon(\pt, \pgb) = \expec X_\epsilon(\pt) Y_\epsilon(\pgb),
\quad 
\beta(\pt, \pgb) = \lim_{\epsilon \to 0} \beta_\epsilon(\pt, \pgb).
\end{align*}
Then (\ref{the_correlations}) implies that 
\begin{align*}
\beta_\epsilon(\pt, \pgb) = \epsilon^{-n/2} (\psi_\epsilon(\pt), w_{\epsilon}(\pgb))_{L^2((0, T) \times \p M)},
\end{align*}
since $\chi_+ = 1$ in $\supp(\psi_\epsilon(\pt))$.
Moreover, Theorem \ref{lem_pairing_asymptotics} implies that $\beta(\pt,\pgb) = \beta$ satisfies (\ref{beta_test}),
and that, for fixed $\pt$ and $\pgb$,
the expected values $\expec|X_\epsilon(\pt)|^2$ and $\expec |Y_\epsilon(\pgb)|^2$ stay bounded as $\epsilon \to 0$.

\subsection{Time translations of the processes}
\label{sec:time_translations}

Let us introduce time translated variables to our analysis. 
We let $j \in \N$, define 
\begin{align*}
\psi_\epsilon^j(t,x;\pt) = \psi_\epsilon(t - (j-1)T, x; \pt)
\end{align*}
and denote by $w_\epsilon^j(t,x; \pgb)$ the solution of 
\begin{align*}
&\p_t^2 w - \Delta_g w = 0 \quad \text{on $(-\infty, jT) \times \R^n$},
\\&w(jT) = w_\epsilon(T;\pgb),\ 
\p_t w_\epsilon(jT) = \p_t w_\epsilon(T;\pgb).
\end{align*}
In the following, we abbreviate $\psi_\epsilon^j(\pt) = \psi_\epsilon^j(\cdot, \cdot ;\pt)$ and $w_{\epsilon}^j(\pgb) = w_{\epsilon}^j(\cdot, \cdot \; ;\pgb)$, respectively.
Moreover, we write
\begin{align*}
X_\epsilon^j(\pt) &= (W, \epsilon^{-n/4} \psi_\epsilon^j(\pt))_{\D' \times C_0^\infty((0, jT) \times \p M)},
\\
Y_\epsilon^j(\pgb) &= (W, \epsilon^{-n/4} \chi_+ w_{\epsilon}^j(\pgb))_{\D' \times C_0^\infty((0, jT) \times \p M)}.
\end{align*}

Notice that for fixed $\pt$ and $\epsilon$ the Gaussian random variables 
$X_\epsilon^j(\pt)$, $j = 1, 2, \dots$, are independent and identically distributed (i.i.d).
Also for fixed $\pgb$ and $\epsilon$ the Gaussian random variables 
\begin{align*}
(W, \epsilon^{-n/4} w_{\epsilon}^j(\pgb))_{\D' \times C_0^\infty(((j-1)T, jT) \times \p M)}, 
\quad j = 1, 2, \dots,
\end{align*}
are i.i.d. Further, we have
\begin{align}
\label{the_inner_product}
&\expec X_\epsilon^j(\pt) Y_\epsilon^j(\pgb) 
= 
\epsilon^{-n/2} (\psi_\epsilon^j(\pt), \chi_+ w_\epsilon^j(\pgb))_{L^2((0, jT) \times \p M)}
\\\nonumber&\quad= 
\epsilon^{-n/2} (\psi_\epsilon(\pt), w_\epsilon(\pgb))_{L^2((0, T) \times \p M)}
= 
\beta_\epsilon(\pt, \pgb),
\end{align}
since $\psi_\epsilon^j$ is supported on $((j-1)T+1, jT) \times \p M$.

Theorem \ref{th_energy_decay} guarantees that $w_\epsilon(\pgb) \in L^2((-\infty,T) \times \p M)$.
Hence for fixed $\pgb$ and $\epsilon > 0$
the random variables
$Y_\epsilon^j(\pgb)$ are uniformly bounded with respect to $j$, that is,
\begin{align}
\label{Y_bounded_wrt_j}
\expec |Y_\epsilon^j(\pgb)|^2 &= \norm{\chi_+ w_\epsilon^j(\pgb)}_{L^2((0, jT) \times \p M)}^2
\\\nonumber&\le \norm{w_\epsilon(\pgb)}_{L^2((-\infty,T) \times \p M)}^2 < \infty.
\end{align}
Furthermore, Theorem \ref{th_energy_decay} implies the following estimate
\begin{align}
\label{tail_estimate}
\norm{\chi_+ w_\epsilon^j(\pgb)}_{L^2((0, (j-J)T) \times \p M)}^2
&\le \norm{w_\epsilon(\pgb)}_{L^2((-\infty,-(J-1)T) \times \p M)}^2 
\\\nonumber&\le C J^{-2n+3}, \quad J \ge 1,
\end{align}
where the constant $C > 0$ does not depend on $j$ and $J$.
Notice that if the dimension is odd we obtain even better estimates. However, in this treatise they are not needed.




\begin{lemma}
	\label{lem:aux_ineq2}
Suppose $X,Y,Z,V : \Omega\to \C$ are zero-mean Gaussian random variables. Then it holds that
\begin{equation}
	|\expec (X Y Z V)| \le 3 \sqrt{\expec |X|^2 \expec |Y|^2 \expec |Z|^2 \expec |V|^2}
\end{equation}
\end{lemma}

\begin{proof}
	Let us begin by pointing out that if $X$ and $Y$ are zero-mean real valued Gaussian random variables, then it holds (see e.g. \cite{LPS2008}) that
	\begin{equation*}
		\expec (XY)^2 = 2 (\expec XY)^2 + \expec X^2 \expec Y^2.
	\end{equation*}
	It follows by a straightforward calculation that if instead $X$ and $Y$ are complex-valued, we have
	\begin{equation}
		\label{eq:aux_ineq2}
		\expec |XY|^2 = |\expec X \bar Y|^2 + |\expec XY|^2 + \expec |X|^2 \expec |Y|^2
	\end{equation}
	and, consequently, $\expec |XY|^2 \leq 3 \expec |X|^2 \expec |Y|^2$
	by the Cauchy--Schwarz inequality. It remains to note that
	\begin{equation*}
		|\expec (X Y Z V)|^2 \leq \expec |XY|^2 \expec |ZV|^2
	 \le 9 \expec |X|^2 \expec |Y|^2 \expec |Z|^2 \expec |V|^2
	\end{equation*}
	and the claim follows.
\end{proof}

\begin{lemma}
	\label{lem:crosscorr_decay}
Let $\epsilon > 0$, $\pt$ and $\pgb$ be fixed and 
let us denote $X^j = X_\epsilon^j(\pt)$ and $Y^j = Y^j_\epsilon(\pgb)$ for $j=1,2,\dots$.
Suppose that $|j-k| \ge 1$. 
Then 
\begin{equation}
	|\expec \left(X^j Y^j-\expec X^j Y^j\right) 
	\overline{\left(X^k Y^k - \expec X^k Y^k\right)}|
	 = \O\left(|j-k|^{-n+\frac 32}\right).
\end{equation}
\end{lemma}
\begin{proof}

\def\l{\text{large}}
\def\s{\text{small}}

We may assume without loss of generality that $j > k$.
Let us split $Y^j$ in two parts 
\begin{align*}
Y^j &= 
(W, \epsilon^{-n/4} \chi \chi_+ w_{\epsilon}^j(\pgb)) + (W, \epsilon^{-n/4} (1 - \chi) \chi_+ w_{\epsilon}^j(\pgb)) 
\\&= Y_\l^j + Y_\s^j,
\end{align*}
where $\chi \in C^\infty(\R)$ satisfies $\chi(t) = 0$ for $t \in (0, kT)$
and $\chi(t) = 1$ outside a neighborhood of $[0,kT]$.
Notice that $(X^j, Y_\l^j)$ and $(X^k, Y^k)$ are independent Gaussian random variables from $\Omega$ to $\C^2$.
Thus also the products $X^j Y_\l^j$ and $X^k Y^k$ are independent random variables.
In particular, 
\begin{align*}
\expec X^j Y_\l^j \overline{X^k Y^k} = \expec X^j Y_\l^j \overline{\expec X^k Y^k}.
\end{align*}
Hence it follows that
\begin{align*}
&|\expec \left(X^j Y^j-\expec X^j Y^j\right) 
	\overline{\left(X^k Y^k - \expec X^k Y^k\right)}|
\\&\quad=
|\expec X^j Y^j \overline{X^k Y^k} - \expec X^j Y^j \overline{\expec X^k Y^k}|
\\&\quad=
|\expec X^j Y_\s^j \overline{X^k Y^k} - \expec X^j Y_\s^j \overline{\expec X^k Y^k}|
\\&\quad\le
4\sqrt{\expec |X^j|^2 \expec |Y_\s^j|^2 \expec |X^k|^2 \expec |Y^k|^2},
\end{align*}
where we have used Lemma \ref{lem:aux_ineq2} and the Cauchy--Schwarz inequality.
The claim follows as $\sqrt{\expec |Y_\s^j|^2} = \O((j-k)^{-n+\frac 32})$
and the norms of $X^j$, $X^k$ and $Y^k$ are uniformly bounded with respect to $j$ and $k$.
\end{proof}

\section{Ergodicity by time translations}
\label{sec:ergodicity}

In our problem formulation the source is an infinitely long realisation of a white noise process. However, the random processes introduced in Section \ref{sec:def_process}
are mostly affected by the source on some finite interval in time. Similarly, the correlation decays fast between the source at other time instances and the variables \eqref{eq:defgbprocess} and \eqref{eq:deftestprocess}, respectively.
In Section \ref{sec:time_translations} we introduced time-translations to sample the full time domain. Below we utilize the zero correlation length of this process \emph{in time} in order to apply ergodicity arguments.
We begin by recording the following lemma.

\begin{lemma}
\label{lem:simple_ineq}
Let $(a_j)_{j=1}^\infty$ and $(b_k)_{k=1}^\infty$ 
be sequences of complex numbers and let us write
\begin{equation*}
	S_N = \frac{1}{N^2} \sum_{1\leq j < k \leq N} a_j b_k.
\end{equation*}
Suppose that $|a_j b_k| \le |j-k|^{-p}$ for $p\geq 1/2$. Then
$|S_N| = \O(N^{-1/2})$.
\end{lemma}

\begin{proof}
We notice 
$\sum_{1\leq j < k \leq N} |j-k|^{-p} = \sum_{\ell=1}^{N-1} (N - \ell) \ell^{-p}$
and conclude that
\begin{equation*}
	N^2|S_N| \le \sum_{\ell=1}^{N-1} (N - \ell) \ell^{-p}
	\le N+\int_1^{N-1} (N-s)s^{-p}ds = \O(N^{3/2}).
\end{equation*}
Now the claim follows.
\end{proof}

\def\e{\epsilon}

Let us now show that $\beta_\e$ and, consequently, $\beta$ can be recovered from the measurement data.
We consider the convergence of the following series
\begin{equation}
	\label{eq:iepsilon_n}
	\I_{\e,N}(\pt, \pgb; \omega) = \frac{1}{N} \sum_{j=1}^{N}
X_\epsilon^j(\pt; \omega)  Y_\epsilon^j(\pgb; \omega).
\end{equation}
\begin{theorem}
	\label{thm:convergence}
	Let $\pt$, $\pgb$ and $\e>0$ be fixed. Then 
	\begin{equation}
		\label{eq:ind=beta}
		\lim_{N\to\infty} \I_{\e,N^3}(\pt, \pgb; \omega) = \beta_\e(\pt, \pgb) 		
	\end{equation}
	almost surely.
\end{theorem}

\begin{proof}
By (\ref{the_inner_product}) we have that $\expec \I_{\e,N}(\pt, \pgb; \omega) = \beta_\e(\pt, \pgb)$.
Let us compute
\begin{multline*}
	{\rm Var}(\I_{\e,N}) =
	\frac{1}{N^2}\sum_{j=1}^N \left( \expec |X^j Y^j|^2-
	|\expec X^j Y^j|^2\right) + \\
	+ \frac{2}{N^2} \sum_{1 \leq j < k \leq N}
	{\rm Re}\ \expec \left((X^j Y^j - \expec X^j Y^j) 
	\overline{(X^k Y^k - \expec X^k Y^k)}\right).
\end{multline*}	
Using the identity 
\eqref{eq:aux_ineq2} and the Cauchy--Schwarz we obtain
\begin{equation*}
	\expec |X^j Y^j|^2-	|\expec X^j Y^j|^2
	= 2\expec |X^j|^2 \expec |Y^j|^2,
\end{equation*}
which is bounded by (\ref{Y_bounded_wrt_j}).
Consequently, Lemmas \ref{lem:crosscorr_decay} and \ref{lem:simple_ineq} guarantee that
${\rm Var}(\I_{\e,N}) = \O(N^{-1/2})$.
It remains to notice that the sum $\sum_{N=1}^\infty {\rm Var}(\I_{\e,N^3})$ converges,
which implies that
$\sum_{N=1}^\infty |\I_{\e,N^3} - \expec \I_{\e,N^3}|^2$ converges almost surely. Therefore
we have almost surely
\begin{align*}
\lim_{N\to\infty} \I_{\e,N^3}(\pt,\pgb) = \lim_{N\to\infty} \expec \I_{\e,N^3}(\pt,\pgb) = \beta_\e(\pt,\pgb).
\end{align*}
\end{proof}

\def\NN{\mathcal N}



Finally, the information obtained by ergodicity arguments
with probability one yields the values of the indicator function $\beta$ by continuity arguments.

\begin{theorem}
	\label{thm:unif}
	A single realization of the measurement $L(F)$ determines the function
	$(\pt, \pgb) \mapsto \I(\pt, \pgb)$ with probability one. 	
\end{theorem}

\begin{proof}
For a fixed $\e>0$, the function $\beta_\epsilon$ is continuous.
Indeed, the inner product 
$(\psi_\epsilon(\pt), w_\epsilon(\pgb))_{L^2((0,T) \times \p M)}$
is clearly continuous with respect to the parameter $\pt$.
Moreover, the formal Gaussian beam $U_\epsilon$ is a smooth function of the variables $(t,x,\pgb)$, see e.g. \cite{Helin2010}.
In particular, the initial data in Theorem \ref{th_from_gb_to_sol} depends smoothly on $\pgb$,
whence the map 
\begin{align*}
\pgb \to w_\epsilon(\pgb) : S (\hat M \setminus M) \to C((0,T); H^1(\R^n))
\end{align*}
is continuous by the energy estimates for the wave equation. 

Let us fix a countable dense set $D_0$ in $D = \left((1,T) \times  \partial_+ SM\right)\times (S (\hat M \setminus M))$. Then with probability one, the equation \eqref{eq:ind=beta} is satisfied simultaneously for all $(\pt,\pgb,\epsilon) \in D_0 \times \{\frac 1 j \; | \; j\in\N\}$.
This together with the continuity of $\beta_\epsilon : D \to \R$ implies that
mappings
$$(\pt, \pgb) \mapsto \I_{1/j}(\pt, \pgb), \quad j\in\N,$$
are recovered. Now, for fixed parameters $\pt \in (1,T) \times  \partial_+ SM$
and $\pgb \in S (\hat M \setminus M)$ we have 
$\beta(\pt,\pgb) = \lim_{j\to\infty} \beta_{1/j}(\pt,\pgb)$
by Theorem \ref{lem_pairing_asymptotics}.
\end{proof}

Theorem \ref{thm_main} follows since the function $\beta$ determines the scattering relation. 

\appendix

\section{White noise as a $H^s_{loc}$-valued random variable}
\label{sec:appendix}

Gaussian white noise has been studied extensively in the literature, see e.g. \cite{Hida_etal}, and
there are different procedures to construct the corresponding probability measure. 
Let us mention the classical approaches by Hida \cite{Hida_original} and the abstract Wiener space construction by Gross \cite{Gross}. It is well-known \cite{Kusuoka82} that the white noise measure is supported on any 
$H^{-s}_{loc}(\R)$, $s >\frac 12$. However, to see whether the classical constructions extend to a Borel measure in this space requires some work. For clarity, we provide a direct construction below. 

Let $D = (1-\p_t^2)^{1/2}$ and write $$H^{-s}(\R,m) =  D^s m^{-\frac 12} L^2(\R),$$
where $m(t) = \frac{1}{1+t^2}$. Clearly, $H^{-s}(\R,m)$ is a separable Hilbert space with the inner product
$(f,g)_{H^{-s}(\R,m)} = (m^{1/2} D^{-s} f, m^{1/2} D^{-s} g)_{L^2(\R)}$.
We refer to \cite{Lofstrom1982} for further properties of $H^{-s}(\R,m)$.

We define $C = D^{-s} m D^{-s}$. 
Let $\{e_j\}_{j\in\N}$ be an orthonormal basis for $L^2(\R)$. Since $D^s m^{-1/2} : L^2(\R) \to H^{-s}(\R,m)$ is an isometry 
we have that functions $f_j = D^s m^{-1/2} e_j$, $j\in \N$, form an orthonormal basis in $H^{-s}(\R,m)$.
By a direct calculation we obtain
$$(C f_j, f_k)_{H^{-s}(\R,m)} = (S e_j,e_k)_{L^2(\R)},$$ where $S=m^{1/2} D^{-2s} m^{1/2}$. The operator $S$ is bounded in $L^2(\R)$ and has a Schwartz kernel
\begin{equation}
	k(x,y) = (1+x^2)^{-\frac 12} (1+y^2)^{-\frac 12} \int_\R \exp(-i(y-x)\xi) (1+\xi^2)^{-s} d\xi.
\end{equation}
Clearly, we have $\int_\R k(x,x) dx < \infty$ when $s > 1/2$ and thus $S$ is trace-class by \cite{Brislawn}. In consequence, also $C$ is trace-class and it follows by \cite[Prop. 2.18]{DaPrato_Zabczyk} that there exists a zero-centered Gaussian measure $\mu_\R$ on $H^{-s}(\R,m)$ with the covariance $C$.

Suppose that $W_\R:\Omega \to H^{-s}(\R,m)$ is a random variable with the distribution $\mu_\R$. For $\phi,\psi \in C^\infty_0(\R)$ we obtain 
\def\duaali{{\mathcal D}' \times {C_0^\infty}}
\begin{align*}
	&\expec ( W_\R, \phi )_{\duaali(\R)} ( W_\R, \psi)_{\duaali(\R)} 
	 \\&\quad= \expec ( W_\R, C^{-1}\phi )_{H^{-s}(\R,m)} ( W_\R, C^{-1}\psi)_{H^{-s}(\R,m)} 
	 \\&\quad=  ( C C^{-1} \phi, C^{-1}\psi)_{H^{-s}(\R,m)} 
	 = ( \phi,\psi)_{L^2(\R)}.
\end{align*}
Hence, $W_\R$ is called the white noise.  

Let us now consider the source studied in this paper. From a construction analogous to the above one, we obtain the white noise measure $\mu_{\p M}$ on $H^{-s_1}(\p M)$, $s_1>(n-1)/2$.
We denote by $\iota$ the embedding of $\p M$ into $\R^n$ 
and use the notation from \cite{Hormander1990} for 
the push-forward $\iota_*$.
We extend $\mu_{\p M}$ to $\R^n$ by defining $\mu_{\R^n} = \iota_* \mu_{\p M}$.
Notice that $\iota_*$ is continuous 
$H^{-s_1}(\p M) \to H^{-s_2}(\R^n)$
where $s_2 = s_1 + 1/2$.
Thus the measure $\mu_{\R^n}$ is Borel on $H^{-s_2}(\R^n)$.
We have the continuous embedding of the tensor product 
\begin{align*}
H^{-s}(\R,m) \otimes H^{-s_2}(\R^n) 
&\hookrightarrow H^{-s_3}(\R^{1+n},\tilde m) 
\hookrightarrow H^{-s_3}_{loc}(\R^{1+n})
\end{align*}
where $\tilde m(t,x) = m(t)$ and $s_3 = s + s_2$. 
Thus $\mu_{\R^{1+n}} = \mu_\R \otimes \mu_{\R^n}$ 
extends to a Borel measure on $H^{-(n+1)/2 - \epsilon}_{loc}(\R^{1+n})$ for any $\epsilon > 0$.
Let us point out that if $W : \Omega \to H^{-(n+1)/2 - \epsilon}_{loc}(\R^{1+n})$
is a random variable with the distribution $\mu_{\R^{1+n}}$ then it satisfies (\ref{the_correlations}).

Finally, we remind the reader that $\supp(W) \subset \R \times \p M$. In particular,
$\chi_+ W$ is almost surely in the space $H^s_{locc}((0,\infty)\times B)$ for
any open set $B$ containing $\p M$. The space
$H^s_{locc}((0,\infty)\times B)$ is defined in Section
\ref{subsec:traces}.

\bibliographystyle{abbrv} 
\bibliography{main}
\end{document}